\documentclass[12pt,english]{amsart}
\usepackage{amsmath,amsthm,amssymb,amsfonts,amscd, array,latexsym, pb-diagram}
\usepackage[noadjust]{cite}
\usepackage[T2A]{fontenc}
\usepackage[cp1251]{inputenc}
\usepackage[english]{babel}
\usepackage{diagmac2} 
\usepackage{tikz}

\usepackage{color} 

\newtheorem{theorem}{Theorem}[section]
\newtheorem{corollary}[theorem]{Corollary}
\newtheorem{lemma}[theorem]{Lemma}
\newtheorem{proposition}[theorem]{Proposition}

\theoremstyle{definition}
\newtheorem{definition}[theorem]{Definition}
\newtheorem{example}[theorem]{Example}
\newtheorem{problem}[theorem]{Problem}

\theoremstyle{remark}
\newtheorem{remark}[theorem]{Remark}

\makeatletter

\renewcommand{\p@enumii}{}

\makeatother

\newcommand{\RR}{\mathbb{R}}

\def\<#1>{\langle #1 \rangle}

\newbox\onebox
\newcommand{\coherent}[1]{\mathbin{\setbox\onebox=\hbox{$=$}\lower0.7\ht%
\onebox\hbox{$\stackrel{#1}{=}$}}}

\newcommand{\acr}{\newline\indent}

\begin{document}

\title[Pseudometric spaces. From minimality to maximality]{Pseudometric spaces. From minimality to maximality in the groups of combinatorial self-similarities}

\author{Viktoriia Bilet}
\address{\textbf{Viktoriia Bilet}\acr
Department of Theory of Functions \acr
Institute of Applied Mathematics and Mechanics of NASU\acr
Dobrovolskogo str. 1, Slovyansk 84100, Ukraine}
\email{viktoriiabilet@gmail.com}

\author{Oleksiy Dovgoshey}
\address{\textbf{Oleksiy Dovgoshey}\acr
Department of Theory of Functions \acr
Institute of Applied Mathematics and Mechanics of NASU \acr
Dobrovolskogo str. 1, Slovyansk 84100, Ukraine \acr
Department of Mathematics and Statistics\acr
University of Turku\acr
Fin-20014, Turku, Finland}

\email{oleksiy.dovgoshey@gmail.com}

\makeatletter
\@namedef{subjclassname@2020}{%
  \textup{2020} Mathematics Subject Classification}
\makeatother


\begin{abstract}
The group of combinatorial self-similarities of a pseudometric space \((X, d)\) is the maximal subgroup of the symmetric group  $\mathbf{Sym} (X)$ whose elements preserve the four-point equality $d(x,y)=d(u,v)$.  Let us denote by \(\mathcal{IP}\) the class of all pseudometric spaces $(X, d)$ for which every combinatorial self-similarity $\Phi\colon~X~\to~X$ satisfies the equality $d(x, \Phi(x))=0,$ but all permutations of metric reflection of $(X, d)$ are combinatorial self-similarities of this reflection.
The structure of \(\mathcal{IP}\) spaces is fully described.

\end{abstract}

\maketitle

\section{Introduction}

The present paper is aimed at describing some interconnections bet\-ween pseudometric spaces, combinatorial similarities, and equivalence relations, begun in \cite{Dov2019IEJA, DovBBMSSS2020, DLAMH2020, BD2206, BD2205}.

The concept of combinatorial similarity for pseudometrics and, more generally, for mappings deffined on the Cartesian square of a set, was introduced in \cite{DLAMH2020}. Theorem 3.1 \cite{DLAMH2020} contains a necessary and sufficient condition under which a mapping is combinatorially similar to a pseudometric. The strongly rigid pseudometrics and discrete ones, were characterized, up to combinatorial similarity, in Theorems 4.13 and 3.9 of \cite{DLAMH2020}, respectively.

The main result of \cite{BD2205}, Theorem 2.8, completely, describes the structure of semimetric spaces $(X,d)$ for which the group of combinatorial self-similarities of $(X,d)$ coincides with the symmetric group $\mathbf{Sym} (X)$.  In particular Theorem 2.8 shows that every permutation  of $X$ is a combinatorial self-similarity of $(X,d)$ if $d$ is discrete or strongly rigid.

The above mentioned results are the starting points of our research.
To describe the general direction of this research we note that the transition from the pseudometric space $(X,d)$ to its metric reflection $(X/{\coherent{0}}, \delta_d)$ naturally generates a homomorphism $H \colon \mathbf{Cs} (X,d) \to \mathbf{Cs} (X/{\coherent{0}}, \delta_d)$ of the groups of combinatorial self-similarities of $(X,d)$ and $(X/{\coherent{0}}, \delta_d)$. How are the algebraic properties of the homomorphism $H$ related to the metric properties of $(X,d)$?

In the present paper we give a metric characterization of pseudometric spaces $(X,d)$ for which the kernel of $H$ coincides with $\mathbf{Cs}(X, d)$ but the equality
$$\mathbf{Cs} (X/{\coherent{0}}, \delta_d) = \mathbf{Sym}(X/{\coherent{0}})$$
holds.

\section{Preliminaries}

Let us start from the classical notion of metric space.

A \textit{metric} on a set \(X\) is a function \(d\colon X^{2} \to \RR\) such that for all \(x\), \(y\), \(z \in X\):
\begin{enumerate}
\item \(d(x,y) \geqslant 0\) with equality if and only if \(x=y\), the \emph{positivity property};
\item \(d(x,y)=d(y,x)\), the \emph{symmetry property};
\item \(d(x, y)\leqslant d(x, z) + d(z, y)\), the \emph{triangle inequality}.
\end{enumerate}

In 1934 \DJ{}uro Kurepa \cite{Kur1934CRASP} introduced the pseudometric spaces which, unlike metric spaces, allow the zero distance between different points.

\begin{definition}\label{ch2:d2}
Let \(X\) be a set and let \(d \colon X^{2} \to \RR\) be a non-negative, symmetric function such that \(d(x, x) = 0\) for every \(x \in X\). The function \(d\) is a \emph{pseudometric} on \(X\) if it satisfies the triangle inequality.
\end{definition}

If \(d\) is a pseudometric on \(X\), we say that \((X, d)\) is a \emph{pseudometric} \emph{space}. We will denote by \(d(X^2)\) the range of the pseudometric \(d\),
\[
d(X^2) := \{d(x, y) \colon x, y \in X\}.
\]

\begin{definition}[\cite{DLAMH2020}]\label{d1.2}
Let \((X, d)\) and \((Y, \rho)\) be pseudometric spaces. The spaces \((X, d)\) and \((Y, \rho)\) are \emph{combinatorially similar} if there exist bijections \(\Psi \colon Y \to X\) and \(f \colon d(X^{2}) \to \rho(Y^{2})\) such that
\begin{equation}\label{d1.2:e1}
\rho(x,y) = f\bigl(d(\Psi(x), \Psi(y))\bigr)
\end{equation}
for all \(x\), \(y \in Y\). In this case, we will say that \(\Psi \colon Y \to X\) is a \emph{combinatorial similarity} and that \((X, d)\) and \((Y, \rho)\) are combinatorially similar pseudometric spaces.
\end{definition}

\begin{example}\label{ex1.3}
If \((X, d)\) and \((Y, \rho)\) are metric spaces, then every iso\-metry  \( X \to Y\) is a combinatorial similarity of \((X, d)\) and \((Y, \rho)\).
\end{example}

Thus the notion of combinatorial similarities can be considered as a generalization of the notion of the isometries of metric spaces.

\begin{remark}\label{r1.4}
The notion of isometry of metric spaces can be extended to pseudometric spaces in various non-equivalent ways. For example, John Kelley \cite{Kelley1965} define the isometries of pseudometric spaces \((X, d)\) and \((Y, \rho)\) as the distance-preserving surjections \(X \to Y\).
\end{remark}

In particular, we also will use the following generalization of isometrics.

\begin{definition}\label{defnew*}\cite{BD2206}
Let $(X, d)$ and $(Y,\rho)$ be pseudometric spaces. A mapping $\Phi: X\to Y$ is a \emph{pseudoisometry} of $(X, d)$ and $(Y, \rho)$ if:
\begin{enumerate}
\item  $\rho(\Phi(x), \Phi(y))=d(x, y)$ holds for all $x, y\in X;$
\item  For every $u\in Y$ there is $v\in X$ such that $\rho(\Phi(v), u)=0.$
\end{enumerate}
\end{definition}

We say that two pseudometric spaces are \emph{pseudoisometric} if there is a pseudoisometry of these spaces.

\smallskip

For every pseudometric spaces $(X, d)$ the set of all combinatorial self-similarities is a group with the function composition as a group operation.
The identity mapping,
\[
\operatorname{Id}_X \colon X \to X, \quad \operatorname{Id}_X(x) = x \text{ for every } x \in X,
\]
is the identity element of this group. If \(d \colon X^2 \to \RR\) is a metric, then we can characterize \(\operatorname{Id}_X\) as a unique mapping \(X \xrightarrow{f} X\) which satisfies the equality
\begin{equation}\label{e1.2}
d(x, f(x)) = 0
\end{equation}
for every \(x \in X\). For the case when the pseudometric \(d \colon X^2 \to \RR\) is not a metric, one can always find a bijection \(X \xrightarrow{f} X\) such that \(f \neq \operatorname{Id}_X\) but \eqref{e1.2} holds for every \(x \in X\).

\begin{example}\label{ex1.4}
Let \((X, d_0)\) be a pseudometric space endowed by the \emph{zero pseudometric}, \(d(x, y) = 0\) for all \(x\), \(y \in X\). Then \eqref{e1.2} holds for every \(x \in X\) and each \(X \xrightarrow{f} X\).
\end{example}

\begin{definition}\label{d1.4}
Let \((X, d)\) be a pseudometric space. A bijection \(f~\colon~X~\to~X\) is a \emph{pseudoidentity} if the equality
\[
d(x, f(x)) = 0
\]
holds for every \(x \in X\).
\end{definition}

\begin{remark}\label{r1.6}
It is clear that, for every pseudometric space \((X, d)\), the set of all pseudoidentities $ X \to X$ is a subgroup of the group of all combinatorial self-similarities of \((X, d)\).
\end{remark}

The combinatorial similarities of pseudometric spaces are the main type of morphisms studied in this paper.

 The groups of all combinatorial self-similarities and all pseudoidentities of a pseudometric space \((X, d)\) will be denoted by $\mathbf{Cs} (X, d)$ and $\mathbf{PI} (X, d)$ respectively. Thus, for every pseudometric space $(X, d)$ we have
 $$
\mathbf{PI} (X, d)\subseteq  \mathbf{Cs} (X, d) \subseteq \mathbf{Sym} (X),
 $$
 where $ \mathbf{Sym} (X)$ is the symmetric group of all permutations of the set $X$.

For every nonempty pseudometric space \((X, d)\), we define a binary relation \(\coherent{0(d)}\) on \(X\) by
\begin{equation}\label{ch2:p1:e1}
(x \coherent{0(d)} y) \Leftrightarrow (d(x, y) = 0),
\end{equation}
for all $x, y\in X.$

In the future, we will simply write \(\coherent{0}\) instead of \(\coherent{0(d)}\), when it is clear which \(d\) we are talking about.

The proof of the following proposition can be found in \cite[Chapter~4, Theorem~15]{Kelley1965}.

\begin{proposition}\label{ch2:p1}
Let \(X\) be a nonempty set and let \(d \colon X^{2} \to \RR\) be a pseudometric on \(X\). Then \({\coherent{0}}\) is an equivalence relation on \(X\) and, in addition, the function \(\delta_d\),
\begin{equation}\label{ch2:p1:e2}
\delta_d(\alpha, \beta) := d(x, y), \quad x \in \alpha \in X/{\coherent{0}}, \quad y \in \beta \in X/{\coherent{0}},
\end{equation}
is a correctly defined metric on the quotient set \(X/{\coherent{0}}\).
\end{proposition}

In what follows we will say that the metric space \((X/{\coherent{0}}, \delta_d)\) is the \textit{metric reflection} of $(X,d).$

\begin{remark}\label{remnew}
It was shown in Theorem~3.3 of \cite{BD2206} that pseudometric spaces $(X, d)$ and $(Y, \rho)$ are pseudoisometric if and only if the metric reflections $(X/{\coherent{0}}, \delta_d)$ and $(Y/{\coherent{0}}, \delta_{\rho})$ are isometric metric spaces.
\end{remark}

Let us define a class $\mathcal{IP}$ of pseudometric spaces as follows.

 \begin{definition}\label{d1.***}
A pseudometric space \((X, d)\) belongs $\mathcal{IP}$ if the equa\-li\-ties
\begin{equation}\label{d1.***-1}
\mathbf{Cs}(X,d) = \mathbf{PI} (X, d)
\end{equation}
and
\begin{equation}\label{d1.***-2}
\mathbf{Cs}(X/ \coherent{0}, \delta_d)= \mathbf{Sym} (X/\coherent{0})
\end{equation}
hold.
\end{definition}


Thus, $(X,d)\in\mathcal{IP}$ holds if and only if the group $\mathbf{Cs}(X,d)$ is as small as possible, but the group $\mathbf{Cs}(X/\coherent{0},\delta_d)$ is as large as possible.

\begin{example}\label{ex1.**}
Let $(X, d)$ be a nonempty metric space.  Then $(X,d)\in \mathcal{IP}$ holds if and only if $|X|=1$. Indeed, the implication
\begin{equation*}
(|X|=1)\Rightarrow ((X, d)\in\mathcal{IP})
\end{equation*}
is evidently valid. Let $(X, d)$ belong to $\mathcal{IP}.$ To prove the equality $|X|=1$ we note that $\mathbf{PI}(X,d)$ contains the mapping $\mathrm{Id}_{X}: X\to X$ only and that $\mathbf{Cs}(X, d)$ and $\mathbf{Cs}(X/ \coherent{0}, \delta_d)$ are isomorphic groups because $(X,d)$ is a metric space. Hence, \eqref{d1.***-1} implies the equality $|\mathbf{Cs}(X/ \coherent{0}, \delta_d)|=1.$ Using the last equality and \eqref{d1.***-2} we obtain $|\mathbf{Sym}(X/ \coherent{0})|=1$ which is possible if and only if $|X/ \coherent{0}|=1.$
Since $d$ is a metric we also have $|X|=|X/ \coherent{0}|.$ The equality $|X|=1$ follows.
\end{example}

The main goal of the paper is to describe the structure of pseudometric spaces belonging to $\mathcal{IP}$. To do this, we introduce into consideration pseudometric generalizations of some well-known classes of metric spaces.

Let \((X, d)\) be a metric space. Recall that the metric \(d\) is said to be \emph{strongly rigid} if, for all \(x\), \(y\), \(u\), \(v \in X\), the condition
\begin{equation}\label{e1.3}
d(x, y) = d(u, v) \neq 0
\end{equation}
implies
\begin{equation}\label{e3.6}
(x = u \text{ and } y = v) \text{ or } (x = v \text{ and } y = u).
\end{equation}
(Some properties of strongly rigid metric spaces are described in \cite{Broughan73, Hattori90, Janos1972, JanosMartin78, Ishiki2022, Martin1977, BDKP2017AASFM, DLAMH2020, DS2021aa, Rouyer2011}.)

The concept of strongly rigid metric can be naturally generalized to the concept of \emph{strongly rigid pseudometric} what was done in paper \cite{DLAMH2020}.

\begin{definition}\label{d4.3}
Let \((X, d)\) be a pseudometric space. The pseudometric \(d\) is \emph{strongly rigid} if every metric
subspace of $(X, d)$ is strongly rigid.
\end{definition}

\begin{remark}
A pseudometric $d: X^{2}\to\mathbb R$ is \emph{strongly rigid} if and only if \eqref{e1.3} implies
\begin{equation}\label{d4.3:e2}
(d(x, u) = d(y, v) = 0) \text{ or } (d(x, v) = d(y, u) = 0)
\end{equation}
for all \(x\), \(y\), \(u\), \(v \in X\).
\end{remark}

\begin{example}\label{r3.2}
The implication \(\eqref{e1.3} \Rightarrow \eqref{d4.3:e2}\) is vacuously true for the zero pseudometric \(d \colon X^{2} \to \RR\). Hence, the zero pseudometric  is strongly rigid.
\end{example}

The well-known example of pseudometric space is a seminormed vector space. Recall that a \emph{semimorm} on a vector space $X$ is a function \linebreak $\| \cdot \| \colon X  \to \mathbb{R}$ such that $\| x+y \| \leqslant \|x\| + \|y\|$ and $\| \alpha x \| = |\alpha| \|x\|$ for all $x, y \in X$ and every scalar  $\alpha$. If $(X, \|\cdot\|)$ is a seminormed vector space then the function $\rho \colon X^2 \to \mathbb{R}$,
$$\rho(x,y) = \|x - y\|,$$ is a pseudometric on $X$.

In the following example we construct a strongly rigid pseudometric subspace of a seminormed real vector space.

\begin{example}\label{r3.2**}
Let $(\mathbb{R}^2, \|\cdot\|)$ be a two-demensional seminormed real vector space endowed with the seminorm $\| \cdot \|$ such that $\| \<x,y>\| = |x|$ for each $\<x,y > \in \mathbb{R}^2$.
The field $\mathbb{R}$ of real numbers can be also considered as a vector space over the field $\mathbb{Q}$ of rational numbers. Let $H$ be a maximal linearly independent over $\mathbb{Q}$ subset of $\mathbb{R}$. If we define a subset $X$ of $\mathbb{R}^2$  as
$$
X=\{ \<x,y > \in \mathbb{R}^2 \colon x \in H \textrm{ and } y \in \mathbb{R}\},
$$
then $X$ is a strongly rigid pseudometric subspace of the seminormed space $(\mathbb{R}^2, \|\cdot \|)$.
\end{example}

We say that a metric $d \colon X^2  \to \mathbb{R}$ is \emph{discrete} if the inequality
\begin{equation*}
|d(X^2)| \leqslant 2
\end{equation*}
holds, where $|d(X^2)|$ is the cardinal number of the set $d(X^2)$.

\begin{remark}\label{rem_1.9_whrn_all}
The standard definition of \emph{discrete metric} can be formulated as:
``The metric on X is discrete if the distance from each point of X to every
other point of X is one.'' (See, for example, \cite[p.~4]{Sea2007}.)
\end{remark}

The following definition is a suitable reformulation of the correspon\-ding concept from \cite{DLAMH2020}.

\begin{definition}\label{defdis} Let $(X, d)$ be a pseudometric space. The pseudometric $d$ is \emph{discrete} if all metric subspaces of $(X, d)$ are discrete. 
\end{definition}

\begin{remark}\label{r3.xx}
If is easy to see that pseudometric \(d \colon X^{2} \to \RR\) is discrete iff $|d(X^{2})| \leqslant 2$.
\end{remark}

\begin{definition}\label{defrect}
A pseudometric space $(X,d)$ is a \emph{pseudorectangle} if all three-point metric subspaces of $(X,d)$ are strongly rigid and isometric and, in addition, there is a four-point metric subspace $Y$ of $(X, d)$ such that for every $x\in X$ we can find $y\in Y$ satisfying $d(x,y)=0.$
\end{definition}

It is easy to see that the metric reflection $(X/\coherent{0},\delta_d)$ of every pseudorectangle $(X, d)$ is a four-point metric space and, in addition, it can be shown that this metric reflection is combinatorially similar to the vertex set of Euclidean non-square rectangle.

\medskip

The paper is organized as follows.

\smallskip

In the next section we recall some known interconnections between equivalence relations, partitions of sets and discrete pseudometrics. 
A simple sufficient condition for the equality $\mathbf{Cs} (X, d)=\mathbf{PI} (X, d)$ is found in Corollary~\ref{cornew} of Proposition~\ref{p1.8}.

The main results of the paper are given in Sections~3 and 4.

A complete description of pseudometric spaces $(X, d)$ that satisfy the equality $\mathbf{Cs}(X/ \coherent{0}, \delta_d)= \mathbf{Sym} (X/\coherent{0})$ is given in Theorem~\ref{Analogt2.4}.
This theorem together with Corollary~\ref{cornew} allow us to characterize $\mathcal{IP}$-spaces in Theorem~\ref{Th3.6}.

 A combinatorial characterization of fibers of pseudometrics is proved in Theorems~\ref{p4.6}, \ref{analogTh} and \ref{thnew} for strongly rigid spaces, discrete spaces and pseudorectangles, respectively. As a corollary of these theorems we obtain a new description of $\mathcal{IP}$-spaces in Theorem~\ref{ThCh4}. In Proposition~\ref{p4.9} we show that pseudorectangles or strongly rigid spaces $(X, d)$ and $(Y, \rho)$ are combinatorially similar if and only if  the binary relations $\coherent{0(d)}$ and $\coherent{0(\rho)}$ are the same. The characteristic properties of $\coherent{0}$ are described in Proposition~\ref{propCh4} for pseudorectangles and strongly rigid pseudometric spaces.

The final result of the paper, Theorem~\ref{ThCh4_4.8}, characterizes the class of discrete pseudometric spaces, strongly rigid pseudometric spaces and pseudorectangles in terms of same extremal properties of these classes.


\section{Partitions of sets}

Let \(U\) be a set. A \emph{binary relation} on \(U\) is a subset of the Cartesian square
\[
U^2 = U \times U = \{\<x, y>\colon x, y \in U\}.
\]
A binary relation \(R \subseteq U^{2}\) is an \emph{equivalence relation} on \(U\) if the following conditions hold for all \(x\), \(y\), \(z \in U\):
\begin{enumerate}
\item \(\<x, x> \in R\), the \emph{reflexivity} property;
\item \((\<x, y> \in R) \Leftrightarrow (\<y, x> \in R)\), the \emph{symmetry} property;
\item \(((\<x, y> \in R) \text{ and } (\<y, z> \in R)) \Rightarrow (\<x, z> \in R)\), the \emph{transitivity} property.
\end{enumerate}


If \(R\) is an equivalence relation on \(U\), then an \emph{equivalence class} is a subset of \(U\) having the form
\begin{equation}\label{e1.1}
[a]_R = \{x \in U \colon \<x, a> \in R\}
\end{equation}
for some \(a \in U\). The \emph{quotient set} \(U / R\) of \(U\) with respect to \(R\) is the set of all equivalence classes \([a]_R\).

Let \(X\) be a nonempty set and \(P = \{X_j \colon j \in J\}\) be a set of nonempty subsets of \(X\). The set \(P\) is a \emph{partition} of \(X\) with the \emph{blocks} \(X_j\), \(j \in J\), if \(\cup_{j \in J} X_j = X\) and \(X_{j_1} \cap X_{j_2} = \varnothing\) for all distinct \(j_1\), \(j_2 \in J\).

\begin{definition}\label{d1.5}
Partitions \(P\) and \(Q\) of a set \(X\) are \emph{equal}, \(P = Q\), if every block of \(P\) is a block of \(Q\) and vice versa.
\end{definition}

Every partition \(P\) of a set \(X\) is a subset of the power set \(2^X\), \(P \subseteq 2^X\), and each block of \(P\) is a point of \(2^X\). Thus, Definition~\ref{d1.5} simply means that \(P = Q\) holds if and only if \(P\) and \(Q\) are the same subsets of \(2^X\). Consequently, \(P = Q\) holds if and only if \(P \subseteq Q\) and \(Q \subseteq P\). The following lemma states that any of the above inclusions is sufficient for \(P = Q\).

\begin{lemma}\label{l2.2}
Let \(P = \{X_j \colon j \in J\}\) and \(Q = \{Y_i \colon i \in I\}\) be partitions of a set \(X\). Then the inclusion \(P \subseteq Q\) (\(Q \subseteq P\)) implies the equality
\begin{equation}\label{l2.2:e1}
P = Q.
\end{equation}
\end{lemma}

\begin{proof}
Let
\begin{equation}\label{l2.2:e2}
P \subseteq Q
\end{equation}
hold. Then, for every \(j_1 \in J\) there is \(i_1 \in I\) such that \(X_{j_1} = Y_{i_1}\). Suppose that inclusion \eqref{l2.2:e2} is strict. Then the set $P$ is a proper subset of the set $Q$. Since every element of $Q$ is a block $Y_i,$ $i \in I,$ there is \(i_0  \in I\) such that
\begin{equation}\label{l2.2:e3}
P \subseteq \{Y_i \colon i \in I \setminus \{i_0\}\}.
\end{equation}
Now from \eqref{l2.2:e3} and the definition of partitions of sets we obtain the contradiction,
\[
X = \bigcup_{j \in J} X_j \subseteq \bigcup_{\substack{i \in I\\ i \neq i_0}} Y_i = X \setminus Y_{i_0} \varsubsetneq X.
\]
Equality \eqref{l2.2:e1} follows.
\end{proof}

There exists the well-known, one-to-one correspondence between the equivalence relations on sets and the partitions of sets (see, for example, \cite[Chapter~II, \S{}~5]{KurMost} or \cite[Theorem~1]{Ore1942}).

\begin{proposition}\label{p1.6}
Let \(X\) be a nonempty set. If \(P = \{X_j \colon j \in J\}\) is a partition of \(X\) and \(R\) is a binary relation on \(X\) such that the logical equivalence
\begin{equation*}
\bigl(\<x, y> \in R\bigr) \Leftrightarrow \bigl(\exists j \in J \colon x \in X_j \text{ and } y \in X_j\bigr)
\end{equation*}
is valid for every \(\<x, y> \in X^{2}\),
then \(R\) is an equivalence relation on \(X\) with the quotient set \(P\) and it is the unique equivalence relation on \(X\) having \(P\) as the quotient set. Conversely, if \(R\) is an equivalence relation on \(X\), then there is the unique partition \(P\) of \(X\) such that \(P\) is the quotient set of \(X\) with respect to \(R\).
\end{proposition}



The next proposition shows, in particular, that if $\Psi: Y\to X$ is a combinatorial similarity of pseudometric spaces $(X, d)$ and $(Y, \rho),$ then the equivalence classes of the relation $\coherent{0(d)}$ are the images of the equivalence classes of $\coherent{0(\rho)}$ under mapping $\Psi.$

\begin{proposition}\label{p1.8}
Let \((X, d)\) and \((Y, \rho)\) be nonempty pseudometric spaces and let $\Psi: Y\to X$ be a combinatorial similarity of these spaces. If $f: d(X^{2})\to\rho (Y^{2})$ is a bijection such that
\begin{equation}\label{p1.8:e0}
\rho(x, y) = f\bigl(d(\Psi(x)), d(\Psi(y))\bigr)
\end{equation}
for all \(x\), \(y \in Y\), then the equalities
\begin{equation}\label{p1.8:e1}
f(0) = 0
\end{equation}
and
\begin{equation}\label{p1.8:e2}
\left(X/ \coherent{0(d)}\right) = \{\Psi(Y_j) \colon Y_j \in Y/ \coherent{0(\rho)}\}
\end{equation}
hold.
\end{proposition}

\begin{proof}
Using Definition~\ref{ch2:d2} and equality \eqref{p1.8:e0} with \(y = x\) we obtain
\[
0 = \rho(x, x) = f\bigl(d(\Psi(x)), d(\Psi(x))\bigr) = f(0)
\]
that implies \eqref{p1.8:e1}.

To prove equality \eqref{p1.8:e2}, we note that \(\{\Psi(Y_j) \colon Y_j \in Y/ \coherent{0(\rho)}\}\) is a partition of the set \(X\), because \(\Psi \colon Y \to X\) is bijective and \(Y/ \coherent{0(\rho)}\) is a partition of \(Y\) by Proposition~\ref{ch2:p1}. Now Proposition~\ref{p1.6} and \eqref{ch2:p1:e1} imply that \eqref{p1.8:e2} holds if and only if
\begin{equation}\label{p1.8:e3}
\bigl(d(x, y) = 0\bigr) \Leftrightarrow \bigl(\rho (\Psi^{-1}(x), \Psi^{-1}(y)) = 0\bigr)
\end{equation}
for all \(x\), \(y \in X\). Logical equivalence \eqref{p1.8:e3} is valid for all \(x\), \(y \in X\) if and only if
\begin{equation}\label{p1.8:e4}
\bigl(d(\Psi(u), \Psi(v)) = 0\bigr) \Leftrightarrow \bigl(\rho(u, v) = 0\bigr)
\end{equation}
for all \(u\), \(v \in Y\). Since \(f \colon d(X^2) \to \rho(Y^2)\) is bijective, equality~\eqref{p1.8:e1} implies that \(d(\Psi(u), \Psi(v)) = 0\) holds if and only if \(f(d(\Psi(u), \Psi(v))) = 0\). Hence, \eqref{p1.8:e3} can be written as
\begin{equation}\label{p1.8:e5}
\bigl(f(d(\Psi(u), \Psi(v))) = 0\bigr) \Leftrightarrow \bigl(\rho(u, v) = 0\bigr).
\end{equation}
Now the validity of \eqref{p1.8:e5} follows from \eqref{p1.8:e0}.
\end{proof}

For the case $(X, d)=(Y, \rho),$ Proposition~\ref{p1.8} implies that the combinatorial self-semilarities preserve the equivalence relation $\coherent{0(d)}.$

\begin{corollary}\label{cor:p1.8}
Let $d$ and $\rho$ be two combinatorially similar pseudometrics defined on the same nonempty set. Then the binary relations $\coherent{0 (d)}$ and $\coherent{0 (\rho)}$ are equal.
\end{corollary}

The next corollary gives a simple sufficient condition under which equality~\eqref{d1.***-1} holds.

\begin{corollary}\label{cornew}
Let $(X, d)$ be a nonempty pseudometric space and let $\{X_{j}: j\in J\}$ be a partition of $X$corresponding the equivalence relation $\coherent{0(d)}.$ If distinct blocks of this partition have different numbers of points, $|X_{j_1}|\ne |X_{j_2}|$ for different $j_1, j_2\in J,$ then the equality
\begin{equation*}
\mathbf{Cs}(X, d)=\mathbf{PI}(X, d)
\end{equation*}
holds.
\end{corollary}
\begin{proof}
Let
\begin{equation}\label{coreq1}
|X_{j_1}|\ne |X_{j_2}|
\end{equation}
holds for all different $j_1, j_2\in J.$ Let us consider an arbitrary combinatorial self-similarity $\Psi: X\to X$ of $(X, d).$ We must show that $\Psi$ is a pseudoidentity of $(X, d).$ To do it, we rewrite equality \eqref{p1.8:e2} in the form
\begin{equation}\label{coreq2}
\{X_j: j\in J\}=\{\Psi(X_j): j\in J\}.
\end{equation}
Since $\Psi$ is a bijective mapping, $|\Psi(X_j)|=|X_j|$ holds for every $j\in J.$ Now, $\Psi\in\mathbf{PI}(X, d)$ follows from \eqref{coreq1} and \eqref{coreq2}.
\end{proof}

The next result directly follows from Proposition~3.6 and Corollary~3.7 of \cite{DLAMH2020}, and shows that a partial converse to Proposition~\ref{ch2:p1} is also true.

\begin{proposition}\label{ch2:p2}
Let \(X\) be a nonempty set and let \(\equiv\) be an equivalence relation on \(X\). Then there is a unique up to combinatorial similarity discrete pseudometric \(d \colon X^{2} \to \RR\) such that
\begin{equation}\label{ch2:p2:e1}
(x \equiv y) \Leftrightarrow (d(x, y) = 0)
\end{equation}
is valid for all \(x\), \(y \in X\).
\end{proposition}

\begin{corollary}\label{cor2.8}
Let \(d\) and $\rho$ be discrete pseudometrics on a set $X.$ Then the following statements are equivalent:
\begin{enumerate}
\item The pseudometric spaces $(X,d) $ and $(X, \rho)$ are combinatorially similar.

\item The binary relation $\coherent{0(d)}$ and $\coherent{0 (\rho)}$ are the same.
\end{enumerate}
\end{corollary}

In the next section of the paper we will prove that the equality of binary relations $\coherent{0(d)}$ and $\coherent{0 (\rho)}$ is equivalent to combinatorial similarity of $(X, d)$ and $(X, \rho),$ when both spaces are pseudorectangles or strongly rigid pseudometric spaces.

\section{Structure of $\mathcal{IP}$ spaces}

The following theorem is a special case of Theorem 2.8 \cite{BD2205}, which gives us a complete description of semimetric spaces satisfying the equality
 $$
 \mathbf{Cs} (X,d)= \mathbf{Sym}(X).
 $$

\begin{theorem}\label{t2.4}
Let \((X, d)\) be a nonempty metric space. Then the following statements are equivalent:
\begin{enumerate}
\item \label{t2.4:s1} At least one of the following conditions has been fulfilled:
\begin{enumerate}
\item \label{t2.4:s1.1} \((X, d)\) is strongly rigid;
\item \label{t2.4:s1.2} \((X, d)\) is discrete;
\item \label{t2.4:s1.3} All three-point subspaces of $(X, d)$ are strongly rigid and isometric.
\end{enumerate}
\item \label{t2.4:s2} The equality   $\mathbf{Cs} (X,d)= \mathbf{Sym}(X)$  holds.
\end{enumerate}
\end{theorem}

Our initial objective is to find a ``pseudometric'' analog of this theorem.

\medskip

The next lemma shows that the transition from pseudometric space to its metric reflection preserves the concept of discreteness, strong rigidity, and the property of being a  pseudorectangle.

 \begin{lemma}\label{l4.xx}
Let \((X,d)\) be a nonempty pseudometric space. Then the following statements hold:
\begin{enumerate}
\item  \((X, d)\) is strongly rigid if and only if \( (X/\coherent{0}, \delta_d) \) is strongly rigid.
\item  \((X, d)\) is discrete if and only if \( (X/\coherent{0}, \delta_d) \) is discrete.
\item  \((X, d)\) is a pseudorectangle if and only if \( (X/\coherent{0}, \delta_d) \) contains exactly four points and all three-point subspaces of \( (X/\coherent{0}, \delta_d) \) are strongly rigid and isometric.
\end{enumerate}
\end{lemma}

\begin{proof}

Let $(X,d)$ be strongly rigid. Then, using the Axiom of Choice (AC), we find a metric subspace $Y$ of the pseudometric space $(X, d)$ such that for every $x\in X$ there is the unique $y\in Y$ which satisfies $d(x, y)=0.$ Since $(X, d)$ is strongly rigid, the metric space $Y$ is also strongly rigid by Definition~\ref{d4.3}. Moreover, Proposition~\ref{ch2:p1} implies that the metric spaces $Y$ and $(X/\coherent{0}, \delta_d)$ are isometric. Hence, $(X/\coherent{0}, \delta_d)$ is strongly rigid.

To complete the proof of statement $(i)$ we must show that the strong rigidity of $(X/\coherent{0}, \delta_d)$ implies that $(X, d)$ is also a strongly rigid.

Let $(X/\coherent{0}, \delta_d)$ be strongly rigid and let $Z$ be a metric subspace of $(X, d).$ Using AC, we find a metric subspace $Y_{Z}$ of $(X, d)$ such that $Y_{Z}\supseteq Z$ and for every $x\in X$ there is the unique $y\in Y_{Z},$ which satisfies $d(x, y)=0.$ Proceeding as above, we can see that $Y_{Z}$ and $(X/\coherent{0}, \delta_d)$ are isometric. Hence, $Y_{Z}$ is strongly rigid and, consequently, being a subspace of $Y_X$, $Z$ is also strongly rigid.

Statement $(i)$ follows.

Let us prove statement $(ii).$
Let $\pi \colon X \to X/\coherent{0} $ be the canonical projection,
$$\pi(x) \colon = \{ y \in X \colon d(x,y) = 0 \}.
$$
Then, by formula \eqref{ch2:p1:e2} of Proposition~\ref{ch2:p1}, the equality
\begin{equation}\label{e4.xx}
d(x,y)= \delta_d (\pi(x), \pi(y) )
\end{equation}
holds for all $x, y \in X.$ Using \eqref{e4.xx}, we see that $d$ and $\delta_d$ has one and  the same range, and, consequently $(ii)$ holds by Remark~\ref{r3.xx}.

The validity of statement $(iii)$ follows from Proposition~\ref{ch2:p1} and Defi\-nition~\ref{d4.3}.
\end{proof}

The following lemma is, in fact, a particular case of Proposition~2.3 from \cite{BD2205}.

\begin{lemma}\label{lem*}
Let $(X, d)$ be a metric space with $|X|\geqslant 4.$ Then the following statements are equivalent:
\begin{enumerate}
\item All three-point subspaces of $(X, d)$ are strongly rigid and isometric.
\item $(X, d)$ is combinatorially similar to the space of vertices of Euclidean non-square rectangle.
\end{enumerate}
\end{lemma}

Now we are ready to characterize the pseudometric spaces satisfying equality \eqref{d1.***-2}.

\begin{theorem}\label{Analogt2.4}
Let $(X, d)$ be a nonempty pseudometric space. Then the following statements are equivalent:
\begin{enumerate}
\item \label{t2.4:s1_2} At least one of the following conditions has been fulfilled:
\begin{enumerate}
\item \((X, d)\) is strongly rigid;
\item \((X, d)\) is discrete;
\item  $(X, d)$ is a pseudorectangle.
\end{enumerate}
\item The equality   \begin{equation}\label{eqCs}\mathbf{Cs}(X/ \coherent{0}, \delta_d)= \mathbf{Sym} (X/\coherent{0})\end{equation}  holds.
\end{enumerate}
\end{theorem}

\begin{proof}
$(i)\Rightarrow (ii).$ Let statement $(i)$ hold. Then, by Lemma~\ref{l4.xx}, at least one from the following statements is valid:

$(s_1)$ \( (X/\coherent{0}, \delta_d) \) is strongly rigid;

\smallskip

$(s_2)$ \( (X/\coherent{0}, \delta_d) \) is discrete;

\smallskip

$(s_3)$ All three-point subspaces of \( (X/\coherent{0}, \delta_d) \) are strongly rigid and isometric and, in addition, $|X/\coherent{0}|=4$ holds.

Now \eqref{eqCs} follows from Theorem~\ref{t2.4}.

$(ii)\Rightarrow (i).$ Let equality \eqref{eqCs} holds. Then, by Theorem~\ref{t2.4}, we have $(s_1)$ or $(s_2),$ or

\smallskip

$(s_4)$ All three-point subspaces of \( (X/\coherent{0}, \delta_d) \) are strongly rigid and discrete.

By Lemma~\ref{l4.xx}, statements $(s_1)$ and $(s_2)$ imply, respectively, statements $(i_1)$ and $(i_2)$ of the theorem being proved. Using Lemma~\ref{lem*}, we see that statements $(s_4)$ and $(s_3)$ are equivalent. Consequently, $(s_4)$ implies $(i_3)$ by Lemma~\ref{l4.xx}.
\end{proof}

The next proposition describes the combinatorial self-similarities of a pseudometric space $(X, d)$ via combinatorial self-similarities of the metric reflection \( (X/\coherent{0}, \delta_d) \) of this space.

\begin{proposition}\label{p3.5}
Let $(X, d)$ be a pseudometric space, let $\Phi: X\to X$ be a bijective mapping and let $\pi: X \to X/\coherent{0}$ be the canonical projection,
\begin{equation*}
\pi(x)=\{y\in X: d(x, y)=0\}.
\end{equation*}
If there is a combinatorial self-similarity $\Psi:X/\coherent{0}\to X/\coherent{0}$ of \( (X/\coherent{0}, \delta_d) \) such that diagram
\begin{equation}\label{eqdiag1*}
\begin{array}{ccc}
X & \xrightarrow{\ \ \pi \ \ \ \ } &
X/\coherent{0} \\
\!\! \!\! \!\! \!\! \! \Phi\Bigg\downarrow &  & \! \!\Bigg\downarrow \Psi
\\
X & \xrightarrow{\ \ \pi \ \ \ \ } & X/\coherent{0}
\end{array}
\end{equation}
is commutative, then $\Phi$ is a combinatorial self-similarity of $(X, d).$
\end{proposition}
\begin{proof}
Let $\Psi$ be a combinatorial self-similarity of \( (X/\coherent{0}, \delta_d) \). Then, using
Definition~\ref{d1.2}, we find a bijection $f: \delta_{d}(X/\coherent{0})^{2}\to\delta_{d}(X/\coherent{0})^{2} $ such that the following diagram
\begin{equation}\label{eqdiag2}
\begin{array}{ccc}
(X/\coherent{0})^{2} & \xrightarrow{\ \ \delta_{d} \ \ \ \ } &
\delta_{d}(X/\coherent{0})^{2} \\
\!\! \!\! \!\! \!\! \! \Psi\otimes\Psi\Bigg\downarrow &  & \! \!\Bigg\downarrow f
\\
(X/\coherent{0})^{2} & \xrightarrow{\ \ \delta_{d} \ \ \ \ } & \delta_{d}(X/\coherent{0})^{2}
\end{array}
\end{equation}
is commutative, where
\begin{equation*}\label{Psi}
\Psi\otimes\Psi(\<a, b>)=\<\Psi(a), \Psi(b)>
\end{equation*}
for every $\<a, b>\in(X/\coherent{0})^{2}.$
Suppose diagram \eqref{eqdiag1*} is commutative. Then
\begin{equation}\label{eqdiag3*}
\begin{array}{ccc}
X^{2} & \xrightarrow{\ \ \pi\otimes\pi \ \ \ \ } &
(X/\coherent{0})^{2} \\
\!\! \!\! \!\! \!\! \! \Phi\otimes\Phi\Bigg\downarrow &  & \! \!\Bigg\downarrow \Psi\otimes\Psi
\\
X^{2} & \xrightarrow{\ \ \pi\otimes\pi \ \ \ \ } & (X/\coherent{0})^{2}
\end{array}
\end{equation}
also is a commutative diagram, where $$\Phi\otimes\Phi(\<x, y>)=\<\Phi(x), \Phi(y)>$$ for every $\<x, y>\in X^{2}.$ The commutativity of \eqref{eqdiag2} and \eqref{eqdiag3*} implies that
\begin{equation}\label{eqdiag4}
\begin{array}{cccccc}
X^{2} & \xrightarrow{\ \ \pi\otimes\pi \ \ \ \ } &
(X/\coherent{0})^{2} & \xrightarrow{\ \ \delta_{d} \ \ \ \ } & \delta_{d}(X/\coherent{0})^{2}\\
\!\! \!\! \!\! \!\! \! \Phi\otimes\Phi\Bigg\downarrow &  & \! \!\Bigg\downarrow \Psi\otimes\Psi & & \! \!\Bigg\downarrow f\\
X^{2} & \xrightarrow{\ \ \pi\otimes\pi \ \ \ \ } & (X/\coherent{0})^{2}& \xrightarrow{\ \ \delta_{d} \ \ \ \ } & \delta_{d}(X/\coherent{0})^{2}\\
\end{array}
\end{equation}
is commutative. By Proposition~\ref{ch2:p1} we have $$\delta_{d}(X/\coherent{0})^{2}=d(X^{2})$$ and, in addition, this proposition implies the equality of mappings $X^{2}\xrightarrow{\ \ d \ \  }d(X^{2})$ and $$X^{2} \xrightarrow{\ \ \pi\otimes\pi \ \  }(X/\coherent{0})^{2} \xrightarrow{\ \ \delta_{d} \ \ }\delta_{d}(X/\coherent{0})^{2}.$$ Hence, the commutativity of \eqref{eqdiag4} gives us the commutativity of the diagram
\begin{equation*}\label{eqdiag3}
\begin{array}{ccc}
X^{2} & \xrightarrow{\ \ d \ \ \ \ } &
d(X^{2}) \\
\!\! \!\! \!\! \!\! \! \Phi\otimes\Phi\Bigg\downarrow &  & \! \!\Bigg\downarrow f
\\
X^{2} & \xrightarrow{\ \ d \ \ \ \ } & d(X^{2}).
\end{array}
\end{equation*}
By Definition~\ref{d1.2}, the last diagram is commutative iff $\Phi: X\to X$ is a combinatorial self-similarity of $(X, d).$
\end{proof}

The next theorem can be considered as one of the main results of the paper.

\begin{theorem}\label{Th3.6}
Let $(X, d)$ be a nonempty pseudometric space and let $\{X_{j}: j\in J\}$ be a partition of $X$corresponding the equivalence relation $\coherent{0(d)}.$ Then $(X, d)\in\mathcal{IP}$ if and only if \begin{equation}\label{Th3.6eq1}|X_{j_1}|\ne |X_{j_2}|\end{equation} holds whenever $j_1, j_2\in J$ are distinct and, in addition, at least one of the following conditions has been fulfilled:
\begin{enumerate}
\item \((X, d)\) is strongly rigid;
\item \((X, d)\) is discrete;
\item  $(X, d)$ is a pseudorectangle.
\end{enumerate}
\end{theorem}
\begin{proof}
Suppose that \eqref{Th3.6eq1} holds whenever $j_1, j_2\in J$ are distinct, and that al least one from conditions $(i)-(iii)$ has been fulfilled. Then the membership $(X, d)\in\mathcal{IP}$ follows from Corollary~\ref{cor:p1.8} and Theorem~\ref{Analogt2.4}.

Let $(X, d)$ belong to $\mathcal{IP}.$ Then the equality
\begin{equation}\label{Th3.6eq2}
\mathbf{Cs}(X/ \coherent{0}, \delta_d)= \mathbf{Sym} (X/\coherent{0})
\end{equation}
holds and, consequently, at least one from conditions $(i)-(iii)$ is valid. Thus, to complete the proof it suffices to show that \eqref{Th3.6eq1} is valid for all distinct $j_1, j_2\in J$.

Suppose contrary that there exist $j_1, j_2\in J$ such that $j_1\ne j_2$ but $|X_{j_1}|= |X_{j_2}|.$ Then there is a bijection $\Phi: X\to X$ which satisfies the equalities
\begin{equation}\label{Th3.6eq3}
\Phi(X_{j_1})=X_{j_2} \quad\mbox{and}\quad \Phi(X_{j})=X_{j}
\end{equation}
whenever $j\in J$ and $j_1\ne j\ne j_2.$

Let $x_{j_1}$ and $x_{j_2}$ be some points of $X_{j_1}$ and $X_{j_2}$ respectively. Write
\begin{equation}\label{Th3.6eq4}
x_{1}^{*}=\pi(x_{j_1})\quad\mbox{and}\quad x_{2}^{*}=\pi(x_{j_2}),
\end{equation}
where $\pi$ is the canonical projection of $X$ on $X/ \coherent{0}$ and define a bijection $\Psi: X/ \coherent{0}\to X/ \coherent{0}$ as
\begin{equation}\label{Th3.6eq5}
\Psi(x): = \begin{cases}
x_{1}^{*} & \text{if } x = x_{2}^{*},\\
x_{2}^{*} & \text{if } x = x_{1}^{*},\\
x & \text{otherwise }.
\end{cases}
\end{equation}
It follows from \eqref{Th3.6eq3} $-$ \eqref{Th3.6eq5} that the diagram
\begin{equation*}\label{eqdiag1}
\begin{array}{ccc}
X & \xrightarrow{\ \ \pi \ \ \ \ } &
X/\coherent{0} \\
\!\! \!\! \!\! \!\! \! \Phi\Bigg\downarrow &  & \! \!\Bigg\downarrow \Psi
\\
X & \xrightarrow{\ \ \pi \ \ \ \ } & X/\coherent{0}
\end{array}
\end{equation*}
is commutative. Moreover, the mapping $\Psi$ is a combinatorial self-similarity of $(X/ \coherent{0}, \delta_d)$ by \eqref{Th3.6eq2}. Hence, $\Phi$ is a combinatorial self-similarity of $(X, d)$ by Proposition~\ref{p3.5}. Definition~\ref{d1.4} and \eqref{Th3.6eq3} imply
\begin{equation*}
\Phi\not\in\mathbf{PI}(X, d).
\end{equation*}
Thus, we have $\Phi \in\mathbf{Cs}(X, d)\setminus \mathbf{PI}(X, d),$ contrary to $(X, d)\in\mathcal{IP}.$
\end{proof}

We conclude this section with the following open problem closely related to Theorem~\ref{Analogt2.4} and Theorem~\ref{Th3.6}.

\begin{problem}
Describe the structure of pseudometric spaces $(X, d)$ for which
$$\mathbf{Cs}(X, d) = \mathbf{PI}(X, d).$$
\end{problem}

\section{From partitions of \(X\) to partitions of \(X^2\)}

It is well-known that for every nonempty set $X$ and arbitrary surjection $f \colon X \to Y$ the family
$$
P_{f^{-1}} \colon=\{ f^{-1}(y) \colon y \in Y \}
$$
is a partition of $X$, where $f^{-1}(y)$ is the inverse image of the singleton $\{y\}$,
$$
f^{-1}(y) \colon= \{x\in X \colon f(x)=y\}.
$$
In what follows we set
\begin{equation}\label{eqch4.1}
P_{d^{-1}}:=\{d^{-1}(t): t\in d(X^{2})\}
\end{equation}
for every nonempty pseudometric space $(X,d).$

In this section we describe the structure of the partition $P_{d^{-1}}$, when $(X, d)$ is strongly rigid or discrete, or $(X, d)$ is a pseudorectangle. This allows us to obtain new characteristics of $\mathcal{IP}$ spaces and expand Corollary~\ref{cor2.8} to strongly rigid pseudometric spaces and pseudorectangles.

\smallskip

The following lemma gives a ``constructive variant'' of Proposition~\ref{p1.6}.

\begin{lemma}\label{l3.1}
Let \(X\) be a nonempty set and let \(P = \{X_j \colon j \in J\}\) be a partition of \(X\). If \(R\) is the equivalence relation corresponding to \(P\), then the equality \(R = \cup_{j \in J} X_{j}^{2}\) holds.
\end{lemma}

For the proof of Lemma~\ref{l3.1} see, for example, Theorem~6 in \cite[p.~9]{Kelley1965}.

\begin{proposition}\label{l3.2}
Let \((X, d)\) be a pseudometric space and $\{X_j \colon j \in J\}$ be the quotient set of $X$ with respect to the equivalence relation \(\coherent{0}\). Then for any fixed non-zero element $t_0$ of $d(X^2)$, the following statements are equivalent:
\begin{enumerate}
\item \label{l3.2:s1} There are different \(j_1\), \(j_2 \in J\) such that
\begin{equation}\label{l3.2:e1}
d^{-1}(t_0) = (X_{j_1} \times X_{j_2}) \cup (X_{j_2} \times X_{j_1}).
\end{equation}
\item \label{l3.2:s2} The assertion
\begin{equation}\label{l3.2:e2}
\bigl((x \coherent{0} u) \text{ and } (y \coherent{0} v)\bigr) \text{ or } \bigl((x \coherent{0} v) \text{ and } (y \coherent{0} u)\bigr)
\end{equation}
is valid whenever
\begin{equation}\label{l3.2:e3}
d(x, y) = t_0 = d(u, v).
\end{equation}
\end{enumerate}
\end{proposition}

\begin{proof}
\(\ref{l3.2:s1} \Rightarrow \ref{l3.2:s2}\). Let  different \(j_1\), \(j_2 \in J\) satisfy \eqref{l3.2:e1}. We must show that \eqref{l3.2:e3} implies \eqref{l3.2:e2} for all \(x\), \(y\), \(u\), \(v \in X\).

Suppose \eqref{l3.2:e3} holds. Then we have
\[
\<x, y>, \<u, v> \in d^{-1}(t_0).
\]
Since the sets \(X_{j_1} \times X_{j_2}\) and \(X_{j_2} \times X_{j_1}\) are disjoint, \eqref{l3.2:e1} implies that only the following cases are possible:
\begin{gather}
\label{l3.2:e4}
\<x, y> \in X_{j_1} \times X_{j_2} \text{ and } \<u, v> \in X_{j_1} \times X_{j_2}, \\
\label{l3.2:e5}
\<x, y> \in X_{j_2} \times X_{j_1} \text{ and } \<u, v> \in X_{j_2} \times X_{j_1}, \\
\label{l3.2:e6}
\<x, y> \in X_{j_1} \times X_{j_2} \text{ and } \<u, v> \in X_{j_2} \times X_{j_1}, \\
\label{l3.2:e7}
\<x, y> \in X_{j_2} \times X_{j_1} \text{ and } \<u, v> \in X_{j_1} \times X_{j_2}.
\end{gather}
The sets \(X_{j_1}\) and \(X_{j_2}\) are different elements of the quotient set \(X / {\coherent{0}}\). Hence, each of \eqref{l3.2:e4} and \eqref{l3.2:e5} implies \(x \coherent{0} u\) and \(y \coherent{0} v\). Analogously, \(x \coherent{0} v\) and \(y \coherent{0} u\) hold whenever \eqref{l3.2:e6} or \eqref{l3.2:e7} is valid. Thus, \ref{l3.2:s2} holds.

\(\ref{l3.2:s2} \Rightarrow \ref{l3.2:s1}\). Let \ref{l3.2:s2} hold and let \(x\), \(y\) be points of \(X\) satisfying
\begin{equation}\label{l3.2:e8}
d(x, y) = t_0.
\end{equation}
Since \(\{X_j \colon j \in J\}\) is the quotient set of \(X\) w.r.t \(\coherent{0}\), there are \(j_1^0\), \(j_2^0 \in J\) such that \(x \in X_{j_1^0}\) and \(y \in X_{j_2^0}\). Equality \eqref{l3.2:e8} and the inequality \(t_0 > 0\) imply that \(j_1^0 \neq j_2^0\).

We claim that \eqref{l3.2:e1} holds with \(j_1 = j_1^0\) and \(j_2 = j_2^0\). Indeed, the inclusion
\[
d^{-1}(t_0) \supseteq (X_{j_1^0} \times X_{j_2^0}) \cup (X_{j_2^0} \times X_{j_1^0})
\]
follows from the triangle inequality, the symmetric property of \(d\) and the definition of \(\coherent{0}\) (see \eqref{ch2:p1:e1}). Hence, to prove \eqref{l3.2:e1}, we must show that the membership
\begin{equation}\label{l3.2:e9}
\<u, v> \in (X_{j_1^0} \times X_{j_2^0}) \cup (X_{j_2^0} \times X_{j_1^0})
\end{equation}
is valid whenever
\begin{equation}\label{l3.2:e10}
d(u, v) = t_0.
\end{equation}
To complete the proof, it suffices to note that \eqref{l3.2:e8} and \eqref{l3.2:e10} imply \eqref{l3.2:e2}, \eqref{l3.2:e3}, and that \eqref{l3.2:e2} implies \eqref{l3.2:e9}.
\end{proof}



Comparing Definition~\ref{d4.3} with statement \ref{l3.2:s2} of Proposition~\ref{l3.2}, we see that \ref{l3.2:s2} can be considered as a singular version of the global property ``to be strongly rigid''. In Theorem~\ref{p4.6} below we characterize the strong rigidity of pseudometrics by ``globalization'' of statement \ref{l3.2:s1} of Proposition~\ref{l3.2}.

Let \(Q = \{X_j \colon j \in J\}\) be a partition of a nonempty set \(X\). Then we define a partition \(Q\,  {\otimes}_1 \, Q \) of \(X^{2}\) by the rule: ``If $|J|=1,$ then \(Q \,  {\otimes}_1 \, Q :=\{X^{2}\},\) otherwise \(B\subseteq X^{2}\) is a block of \(Q \,  {\otimes}_1 \, Q\) if and only if either \linebreak \(B = \bigcup_{j \in J} X_{j}^{2}\) or there are \emph{distinct} \(j_1\), \(j_2 \in J\) such that \(B = (X_{j_1} \times X_{j_2}) \cup (X_{j_2} \times X_{j_1})\)''.

The next theorem follows directly from Theorem~4.13 and Corollary~4.14 of \cite{DLAMH2020}.

\begin{theorem}\label{p4.6}
Let \((X, d)\) be a nonempty pseudometric space.
Then the following conditions are equivalent:
\begin{enumerate}
\item\label{p4.6:s1} \(d \colon X^{2} \to \RR\) is strongly rigid.
\item\label{p4.6:s3} If $Q$ is a partition corresponding to $\coherent{0},$ then $Q \,  {\otimes}_1 \, Q $ and $P_{d^{-1}}$ are equal,
    \begin{equation}\label{eqvspom}
    Q \,  {\otimes}_1 \, Q = P_{d^{-1}}.
    \end{equation}
\item\label{p4.6:s2} There is a partition \(Q\) of \(X\) such that \eqref{eqvspom} holds.
\end{enumerate}
\end{theorem}

Similarly to $Q \,  {\otimes}_1 \, Q $ for every partition of $Q=\{X_{j}: j\in J\}$ of $X$ we define $Q \,  {\otimes}_2 \, Q$ as: ``If $|J|=1,$ then $Q \,  {\otimes}_2 \, Q :=\{X^{2}\},$ otherwise a set $B\subseteq X^{2}$ is a block of $Q \,  {\otimes}_2 \, Q $ if and only if either $B= \bigcup_{j\in J}X_{j}^{2}$ or $B=X^{2}\setminus \bigcup_{j\in J}X_{j}^{2}.$''

The following result is an analog of Theorem~\ref{p4.6} for discrete pseudometrics.

\begin{theorem}\label{analogTh}
Let \((X, d)\) be a nonempty pseudometric space.
Then the following conditions are equivalent:
\begin{enumerate}
\item\label{p4.6:s1_2} \(d \colon X^{2} \to \RR\) is discrete.
\item If $Q$ is a partition corresponding to $\coherent{0},$ then $Q \,  {\otimes}_2 \, Q $ and $P_{d^{-1}}$ are equal,
    \begin{equation}\label{eqvspom1}
    Q \,  {\otimes}_2 \, Q  = P_{d^{-1}}.
    \end{equation}
\item\label{p4.6:s2_2} There is a partition \(Q\) of \(X\) such that \eqref{eqvspom1} holds.
\end{enumerate}
\end{theorem}

\begin{proof}
The implications $(i)\Rightarrow (ii)\Rightarrow (iii)\Rightarrow (i)$ are evidently valid if $d$ is the zero pseudometric on $X.$ Let us consider the case when $|d(X^{2})|\geqslant 2.$

$(i)\Rightarrow (ii).$ Let $d$ be a discrete pseudometric and let $$Q=\{X_{j}: j\in J\}$$ be the partition of $X$ corresponding to the relation ${\coherent{0(d)}}$. By Lemma~\ref{l3.1}, the equality
\begin{equation*}
\left(\coherent{0(d)}\right)=\bigcup_{j\in J}X_{j}^{2}
\end{equation*}
holds. Using \eqref{ch2:p1:e1} we see that
\begin{equation}\label{t2.4eq}
d^{-1}(0)=\bigcup_{j\in J} X_{j}^{2}.
\end{equation}
It was noted in Remark~\ref{r3.xx} that the inequality $|d(X^{2})|\leqslant 2$ holds for discrete $d.$ The last inequality and $|d(X^{2})|\geqslant 2$ imply that $|d(X^{2})| = 2.$ Thus, the partition $P_{d^{-1}}$ of $X^{2}$ contains exactly two blocks. Since one of this block is $d^{-1}(0),$ the second one coincides with $X^{2}\setminus \bigcup\limits_{j\in J}X_{j}^{2}$ by \eqref{t2.4eq}. Now $(ii)$ follows from the definition of $Q \,  {\otimes}_2 \, Q .$

$(ii)\Rightarrow (iii).$ This implication is trivially valid.

$(iii)\Rightarrow (i).$ Let a partition $Q=\{X_{j}: j\in J\}$ of $X$ satisfy equality \eqref{eqvspom1}. Then this equality and the definition of $Q \,  {\otimes}_2 \, Q $ imply $|d(X^{2})|~=~2.$ Using Remark~\ref{r3.xx}, we see that $d$ is a discrete pseudometric.
\end{proof}

\begin{remark}
Theorem~\ref{analogTh} can be considered as a special case of Theo\-rem~3.9 from \cite{DLAMH2020}, that describes all mappings with domain $X^{2}$ which are combinatorially similar to discrete pseudometrics on $X.$
\end{remark}

Let $X$ be a set with $|X|\geqslant 4$ and let $Q=\{X_1, X_2, X_3, X_4\}$ be a partition of the set $X$. Let us denote by $Q \,  {\otimes}_3 \, Q $ a partition of $X^{2}$ having the blocks: $\mathop{\bigcup}\limits_{i=1, 4} X_{i}^{2},$

\begin{equation}\label{eq*}(X_1\times X_2)\cup(X_2\times X_1)\cup(X_3\times X_4)\cup(X_4\times X_3),\end{equation}

\begin{equation}\label{eq**}(X_1\times X_3)\cup(X_3\times X_1)\cup(X_2\times X_4)\cup(X_4\times X_2),\end{equation}

\begin{equation}\label{eq***}(X_1\times X_4)\cup(X_4\times X_1)\cup(X_2\times X_3)\cup(X_3\times X_2).\end{equation}

\begin{theorem}\label{thnew}
Let \((X, d)\) be a nonempty pseudometric space.
Then the following conditions are equivalent:
\begin{enumerate}
\item\label{thnew:s1} $(X ,d)$ is a pseudorectangle.
\item If $Q$ is a partition corresponding to $\coherent{0},$ then $Q \,  {\otimes}_3 \, Q $ and $P_{d^{-1}}$ are equal,
    \begin{equation}\label{eqvspom2}
    Q \,  {\otimes}_3 \, Q  = P_{d^{-1}}.
    \end{equation}
\item\label{thnew:s2} There is a partition \(Q\) of \(X\) such that \eqref{eqvspom2} holds.
\end{enumerate}
\end{theorem}

\begin{proof}
$(i)\Rightarrow (ii).$ Let $(X, d)$ be a pseudorectangle. 

By Definition~\ref{defrect}, there is a set $Y=\{y_1, y_2, y_3, y_4\}$, such that $Y\subseteq X$ and the sets $[y_i]_{\coherent{0}},$
\begin{equation*}\label{thnew:eq2*}
[y_i]_{\coherent{0}}=\{x\in X: d(x, y_i)=0\}, \, i=1, ..., 4,
\end{equation*}
are the equivalence classes of the relation $\coherent{0}.$

We claim that \eqref{eqvspom2} holds if
\begin{equation}\label{thnew:eq1}
Q=\{X_1, X_2, X_3, X_4\}
\end{equation}
with
\begin{equation}\label{thnew:eq3}
X_{i}=[y_i]_{\coherent{0}},\, i=1, ..., 4.
\end{equation}
To prove the last claim  we first show that $P_{d^{-1}}$ contains exactly four blocks, i.e.
\begin{equation}\label{thnew:eq4}
|d(X^{2})|=4
\end{equation}
holds.

Indeed, by Definition~\ref{defrect}, all three-point metric subspaces of $(X, d)$ are isometric that implies
\begin{equation}\label{thnew:eq5}
|d(X^{2})|\leqslant 4.
\end{equation}
Moreover, it is easy to see that a finite nonempty metric space $(Z, \rho)$ is strongly rigid if and only if the number of two-point subsets of $Z$ is the same as the number of non-zero elements of $\rho(Z^{2}),$
\begin{equation*}
|\rho(Z^{2})|=\frac{|Z||Z-1|}{2}+1.
\end{equation*}
In particular, a three-point metric subspace $S$ of the pseudometric space $(X, d)$ is strongly rigid if and only if $|d(S^{2})|=4.$ Consequently, Definition~\ref{defrect} implies
 \begin{equation}\label{thnew:eq6}
|d(X^{2})|\geqslant |d(S^{2})|=4
\end{equation}
whenever $S$ is a three-point metric subspace of $(X, d).$ Now, equality \eqref{thnew:eq4} follows from \eqref{thnew:eq5} and \eqref{thnew:eq6}.

Let us prove \eqref{eqvspom2}. By Lemma~\ref{l2.2}, it suffices to show that the inclusion
\begin{equation*}
P_{d^{-1}}\subseteq Q \,  {\otimes}_3 \, Q
\end{equation*}
holds, i.e., that
\begin{equation}\label{thnew:eq7}
d^{-1}(t)\in Q \,  {\otimes}_3 \, Q
\end{equation}
is valid for every $t\in d(X^{2}).$

Let $t_1\in d(X^{2})$ and $t_1>0$ hold. Then there exist two different points $y_{i_1}, y_{i_2}\in\{y_1, y_2, y_3, y_4\}$ such that $d(y_{i_1}, y_{i_2})=t_1.$ Without loss the generality, we assume that $i_1=1$ and $i_2=2.$ Then, by \eqref{thnew:eq3}, we have $$ y_{i_1}=y_1\in X_1 \quad\mbox{and}\quad y_{i_2}=y_2\in X_2,$$
 that implies
\begin{equation}\label{thnew:eq8}
d^{-1}(t_1)\supseteq (X_1\times X_2)\cup (X_2\times X_1).
\end{equation}
Since all triangles of the metric space $\{y_1, y_2, y_3, y_4\}$ are isometric and strongly rigid, Lemma~\ref{lem*} implies the equality
 \begin{equation}\label{thnew:eq9}
d(y_3, y_4)=t_1
\end{equation}
and, in addition,
\begin{equation}\label{thnew:eq10}
d(y_1, y_3)=d(y_2, y_4)\ne t_1\ne d(y_1, y_4)=d(y_2, y_3).
\end{equation}
Similarly \eqref{thnew:eq8}, equality \eqref{thnew:eq9} implies
\begin{equation*}
d^{-1}(t_1)\supseteq (X_3\times X_4)\cup (X_4\times X_3),
\end{equation*}
and, consequently,
\begin{equation*}
d^{-1}(t_1)\supseteq (X_1\times X_2)\cup (X_2\times X_1)\cup (X_3\times X_4)\cup (X_4\times X_3).
\end{equation*}
If the last inclusion is strict, then there are points $x_1, x_2\in X$ such that $d(x_1, x_2)=t_1$ and
\begin{equation*}
\<x_1, x_2> \in (X_1\times X_3)\cup (X_3\times X_1)\cup (X_2\times X_4)\cup (X_4\times X_2)\cup
\end{equation*}
\begin{equation*}
(X_1\times X_4)\cup (X_4\times X_1)\cup (X_2\times X_3)\cup (X_3\times X_2)
\end{equation*}
are satisfied, contrary to \eqref{thnew:eq10}. Thus, the equality
\begin{equation*}
d^{-1}(t_1)=(X_1\times X_2)\cup (X_2\times X_1)\cup (X_3\times X_4)\cup (X_4\times X_3)
\end{equation*}
holds, that together with
\begin{equation*}
d^{-1}(0)=\bigcup_{i=1}^{4}X_{i}^{2}
\end{equation*}
implies \eqref{thnew:eq7} for every $t\in d(X^{2}).$

$(ii)\Rightarrow (iii).$ This implication is trivially valid.

$(iii)\Rightarrow (i).$ Let equality \eqref{eqvspom2} hold with
$Q=\{X_1, X_2, X_3, X_4\}.$ We must prove that $(X, d)$ is a pseudorectangle.

As in the proof of Theorem~\ref{p4.6}, it can be shown that
\begin{equation}\label{thnew:eq11}
d^{-1}(0)=\bigcup_{i=1}^{4}X_{i}^{2}.
\end{equation}
Moreover, the definitions of $Q \,  {\otimes}_3 \, Q $ and \eqref{eqvspom2} imply the equality
\begin{equation}\label{thnew:eq12}
|d(X^{2})|=4.
\end{equation}
Let us consider a four-point set $Y=\{y_1, y_2, y_3, y_4\}$ such that $y_i\in X_i$ holds for every $i\in\{1, 2, 3, 4\}.$ Then, using \eqref{thnew:eq11}, we see that $Y$ is a four-point metric subspace of $(X, d)$ and, for every $x\in X$, there is $y\in Y$ such that $d(x, y)=0.$ Now, by Definition~\ref{defrect}, $(X, d)$ is a pseudorectangle if and only if all three-point metric subspaces of $(X, d)$ are strongly rigid and isometric.

Hence, using \eqref{thnew:eq12}, we see that $(X, d)$ is a pseudorectangle if and only if
\begin{equation}\label{thnew:eq13}
|d(Z^{2})|=4
\end{equation}
holds for every three-point metric subspace $Z$ of $(X, d).$
Let us consider arbitrary $z_1, z_2, z_3\in X$ such that $d(z_i, z_j)\ne 0$ for all distinct $i, j\in\{1,2, 3\}.$ Since every permutation
\begin{equation*}
\{X_1, X_2, X_3, X_4\}\rightarrow\{X_1, X_2, X_3, X_4\}
\end{equation*}
preserves the partition $Q \,  {\otimes}_3 \, Q $ of the set $X^{2}$, we can assume that $z_1\in X_1,$ $z_2\in X_2$ and $z_3\in X_3.$ Now \eqref{thnew:eq13} follows from \eqref{eq*}, \eqref{eq**} and \eqref{eq***}.
\end{proof}

Using Theorem~\ref{p4.6} and Theorems~\ref{analogTh} and \ref{thnew}, we obtain the following modification of Theorem~\ref{Th3.6}.

\begin{theorem}\label{ThCh4}
Let \((X, d)\) be a nonempty pseudometric space and let
\begin{equation*}
Q=\{X_{j}: j\in J\}
\end{equation*}
be a partition of $X$ corresponding to the equivalence relation $\coherent{0(d)}.$ Then $(X,d)\in\mathcal{IP}$ if and only if
\begin{equation*}
P_{d^{-1}}\in \{Q  \,  {\otimes}_1 \,Q, Q \,  {\otimes}_2 \, Q, Q \,  {\otimes}_3 \, Q \}
\end{equation*}
and
\begin{equation*}
|X_{j_1}|\ne |X_{j_2}|
\end{equation*}
holds whenever $j_1, j_2\in J$ are distinct.
\end{theorem}

Let us extend Corollary~\ref{cor2.8} to pseudorectangles and strongly rigid pseudometric spaces. 

\begin{lemma}\label{Lemch4}
Let $(X, d)$ and $(X, \rho)$ be nonempty pseudometric spaces. If the equality
\begin{equation}\label{eqch4.2}
P_{d^{-1}}=P_{\rho^{-1}}
\end{equation}
holds, then the identical mapping $Id_{X}: X\to X$ is a combinatorial similarity of $(X, d)$ and $(X, \rho).$
\end{lemma}
\begin{proof}
Let \eqref{eqch4.2} hold. Then by \eqref{eqch4.1} we have
\begin{equation*}
\{d^{-1}(t): t\in d(X^{2})\}=\{\rho^{-1}(\tau): \tau\in \rho(X^{2})\}
\end{equation*}
and, consequently, there is a bijection $f:d(X^{2})\to \rho(X^{2})$ such that
\begin{equation}\label{eqch4.3}
f(t)=\tau \quad\mbox{if and only if}\quad \rho^{-1}(\tau)=d^{-1}(t)
\end{equation}
whenever $t\in d(X^{2}),$ $\tau\in\rho(X^{2}).$

Using \eqref{eqch4.3} it is easy to see that the diagram
\begin{equation}\label{eqch4.4}
\begin{array}{ccc}
X^{2} & \xrightarrow{\ \ d \ \ \ \ } &
d(X^{2}) \\
\!\! \!\! \!\! \!\! \! Id_{X^{2}}\Bigg\downarrow &  & \! \!\Bigg\downarrow f
\\
X^{2} & \xrightarrow{\ \ \rho \ \ \ \ } & \rho(X^{2}).
\end{array}
\end{equation}
is commutative, when $Id_{X^{2}}: X^{2}\to X^{2}$ is the identical mapping of $X^{2}.$

The mapping $Id_{X^{2}}$ coincides with the mapping $Id_{X}\otimes Id_{X},$
\begin{equation*}
Id_{X^{2}}(\< x, y>)=\<x, y>=\<Id_{X} (x), Id_{X}(y)>
\end{equation*}
holds for every $\< x, y>\in X^{2}.$ Thus the commutativity of \eqref{eqch4.4} implies the commutativity of
\begin{equation*}
\begin{array}{ccc}
X^{2} & \xrightarrow{\ \ d \ \ \ \ } &
d(X^{2}) \\
\!\! \!\! \!\! \!\! \! Id_{X}\otimes Id_{X}\Bigg\downarrow &  & \! \!\Bigg\downarrow f
\\
X^{2} & \xrightarrow{\ \ \rho \ \ \ \ } & \rho(X^{2}).
\end{array}
\end{equation*}
By Definition~\ref{d1.2}, the last diagram is commutative if and only if the mapping $Id_{X}: X\to X$ is a combinatorial similarity of $(X, d)$ and $(X, \rho).$
\end{proof}


\begin{proposition}\label{p4.9}
Let $(X, d)$ and $(X, \rho)$ be either pseudorectangles or nonempty strongly rigid pseudometric spaces. Then the following statements are equivalent:
\begin{enumerate}
\item The pseudometric spaces $(X,d)$ and $(X, \rho)$ are combinatorially similar.

\item The binary relations $\coherent{0(d)}$ and $\coherent{0 (\rho)}$ are the same.
\end{enumerate}
\end{proposition}
\begin{proof}
$(i)\Rightarrow (ii).$ The validity of this implication follows from Corollary~\ref{cor:p1.8}.

$(ii)\Rightarrow (i).$ Let $(ii)$ hold. Suppose that both spaces $(X,d)$ and $(X, \rho)$ are strongly rigid.

Let $Q=\{X_{j}: j\in J\}$ be the partition of $X$ corresponding the equivalence relation $\coherent{0(d)}.$ Then, by Theorem~\ref{p4.6}, we have
\begin{equation}\label{eqch4.5}
P_{d^{-1}}=Q \,  {\otimes}_1 \, Q .
\end{equation}
Since the relations $\coherent{0(d)}$ and $\coherent{0 (\rho)}$ are the same, we have
\begin{equation}\label{eqch4.6}
P_{\rho^{-1}}=Q  \,  {\otimes}_1 \, Q.
\end{equation}
Equalities \eqref{eqch4.5} and \eqref{eqch4.6} imply the equality
\begin{equation*}
P_{d^{-1}} = P_{\rho^{-1}}.
\end{equation*}
Consequently, $(X, d)$ and $(X, \rho)$ are combinatorially similar by Lemma~\ref{Lemch4}.

For the case when $(X, d)$ and $(X, \rho)$ are pseudorectangles the validity of $(ii)\Rightarrow (i)$ can be proved similarly if apply Theorem~\ref{thnew} instead of Theorem~\ref{p4.6} and $Q \,  {\otimes}_3 \, Q $ instead of $Q\,  {\otimes}_1 \,Q.$
\end{proof}


In the next proposition we denote by $\mathfrak{c}$ the cardinality of the continu\-um.

\begin{proposition}\label{propCh4}
Let $X$ be a nonempty set, let $\equiv$ be an equivalence relation on $X$ and  $$Q=\{X_{j}: j\in J\}$$ be the partition of $X$ corresponding to $\equiv$. Then the following statements hold:
\begin{enumerate}
\item The inequality $|J|\leqslant\mathfrak{c}$ holds if and only if there is a strongly rigid pseudometric space $(X,d)$ such that
    \begin{equation}\label{eqch4D}
    (x\equiv y)\Leftrightarrow (d(x, y)=0)
    \end{equation}
is valid for all $x, y\in X.$
\item The equality $|J|=4$ holds if and only if there is a pseudorec\-tangle $(X,d)$ such that \eqref{eqch4D} is valid for all $x, y\in X.$
\end{enumerate}
\end{proposition}
\begin{proof}
Statement $(i)$ was proved in Theorem~4.13 of \cite{DLAMH2020}. Let us prove the validity of $(ii)$.

Let $|J|=4$ hold. Let us consider an injective mapping $f:~Q~\,  {\otimes}_3 \,Q~\to~\mathbb R$  such that
\begin{equation}\label{eqch4.7}
f\left(\bigcup_{j\in J}X_{j}^{2}\right)=0
\end{equation}
and
\begin{equation}\label{eqch4.8}f(B)\in (1, 2),\end{equation}
whenever $B$ is a block of $Q\,  {\otimes}_3 \,Q $ defined by equalities \eqref{eq*} $-$ \eqref{eq***}.

Write $\equiv_3$ for the equivalence relation corresponding to the partition $Q\,  {\otimes}_3 \, Q$ and denote by $d$ the mapping
\begin{equation*}
\begin{array}{ccccc}
X^{2} & \xrightarrow{\ \pi\otimes\pi \ \ \ \ } &
Q\,  {\otimes}_3 \, Q & \xrightarrow{\ \ } \mathbb R, \\
\end{array}
\end{equation*}
where $\pi$ and $\pi \otimes \pi$ are, respectively, the canonical projection of $X$ on $X/\equiv$ and $X^2$ on $X^2/\equiv_3$. By Theorem~\ref{p4.6}, $(X,d)$ is a pseudorectangle. (We note only that \eqref{eqch4.8} implies the triangle inequality for $d$.) The definition of $d$ and Lemma~\ref{l3.1} imply that $Q$ is the partition of $X$ corresponding the equivalence relation $\coherent{0(d)}.$

Let us consider a pseudorectangle $(X, d)$ such that
$Q=\{X_{j}: j\in J\}$ is a partition corresponding to the equivalence relation $\coherent{0(d)}.$ Then the equality $|J|=4$ follows from Theorem~\ref{thnew}.
\end{proof}




\begin{remark}\label{rem*}
By Proposition~\ref{p4.9}, all strongly rigid pseudometric spaces (pseudorectangles) $(X, d)$ satisfying \eqref{eqch4D} are combinatorially similar. Thus, Proposition~\ref{propCh4} can be considered as a development of Proposition~\ref{ch2:p2}.
\end{remark}

\begin{theorem}\label{ThCh4_4.8}
Let $\mathcal{CL}$ be a maximal class of nonempty pseudometric spaces such that for every $(X, d)\in \mathcal{CL}$ and each pseudometric space $(Y, \rho)$ we have:
\begin{enumerate}
\item[$(i_1)$] $(Y, \rho)\in \mathcal{CL}$ whenever $(X, d)$ and $(Y, \rho)$ are pseudoisometric.

\item[$(i_2)$] If $(Y, \rho)\in\mathcal{CL},$ and $Y=X,$ and the relations  $\coherent{0(d)}$ and $\coherent{0 (\rho)}$ are the same, then the identical mapping $Id_{X}: X\to X$ is a combinatorial similarity of $(X, d)$ and $(Y, \rho).$

\item[$(i_3)$] Every nonempty subspace of $(X, d)$ belongs to $\mathcal{CL}.$
\end{enumerate}
Then exactly one from the following statements holds.
\begin{enumerate}
\item[$(ii_1)$] $\mathcal{CL}$ is the class of all strongly rigid pseudometric spaces.

\item[$(ii_2)$] $\mathcal{CL}$ is the class of all discrete pseudometric spaces.

\item[$(ii_3)$] $\mathcal{CL}$ is the union of the class of all pseudorectangles with the class of all strongly rigid pseudometric spaces $(X, d)$ satisfying the inequality $|d(X^{2})|\leqslant 4.$
\end{enumerate}
\end{theorem}

\begin{remark}\label{remch4}
The maximality of $\mathcal{CL}$ means that for every class $\mathcal{CL}^{o}$ of nonempty pseudometric spaces, the inclusion
$\mathcal{CL}^{o}\supseteq\mathcal{CL}$ implies the equality $\mathcal{CL}^{o}=\mathcal{CL}$ whenever $\mathcal{CL}^{o}$ satisfies conditions $(i_1) - (i_3)$ for every $(X, d)\in\mathcal{CL}^{o}$ and every pseudometric space $(Y, \rho).$
\end{remark}

\noindent\emph{Proof of Theorem~\ref{ThCh4_4.8}.}
Let $\mathcal{CL}^{*}$ be an arbitrary class of nonempty pseudometric spaces such that $(i_1) - (i_3)$ are valid with $\mathcal{CL}=\mathcal{CL}^{*}$ for every $(X, d)\in\mathcal{CL}^{*}$ and each pseudometric space $(Y, \rho).$

Let us consider an arbitrary $(X, d)\in\mathcal{CL}^{*}$ and let $(X/\coherent{0}, \delta_d)$ be the metric reflection of $(X, d)$. It follows directly from Definition~\ref{defnew*} and Proposition~\ref{ch2:p1} that the canonical projection $\pi: X\to X/\coherent{0}$ is a pseudoisometry of $(X, d)$ and $(X/\coherent{0}, \delta_d).$ Hence, $(X/\coherent{0}, \delta_d)$ belongs to $\mathcal{CL}^{*}$ by condition $(i_1).$ Since $(X/\coherent{0}, \delta_d)$ is a metric space, the relation $\coherent{0(\delta_d)}$ is the identical relation on $X/\coherent{0},$ i.e., for all $a, b\in X/\coherent{0},$ $\<a, b>\in \coherent{0(\delta_d)}$ holds if and only if $a=b.$

Let $\Phi:X/\coherent{0}\to X/\coherent{0}$ be an arbitrary bijection of $X/\coherent{0}.$ The function $\rho^{\Phi}: (X/\coherent{0})^{2}\to\mathbb R$ defined as
\begin{equation}\label{eqch4.9}
\rho^{\Phi}(a, b)=\delta_{d}(\Phi(a), \Phi(b))
\end{equation}
is a metric on $X/\coherent{0}$ and $\Phi$ is an isometry of the metric spaces $(X/\coherent{0}, \delta_d)$ and $(X/\coherent{0}, \rho^{\Phi}).$ Since every isometry is a pseudoisometry, the membership $(X/\coherent{0}, \delta_d)\in\mathcal{CL}^{*}$ implies $(X/\coherent{0}, \rho^{\Phi})\in\mathcal{CL}^{*}$ by condition $(i_2).$ Furthermore, we have
\begin{equation*}
\left(\<a, b>\in \coherent{0(\rho)}\right)\Leftrightarrow \left(\rho^{\Phi}(a, b)=0\right)
\end{equation*}
for all $a,b\in X/\coherent{0},$ because $\rho$ is a metric on $X/\coherent{0}.$ Thus, the relations $\coherent{0(\delta_d)}$ and $\coherent{0(\rho^{\Phi})}$ are the same. Consequently, by condition $(i_2),$ the identical mapping of $X/\coherent{0}$ is a combinatorial similarity of $(X/\coherent{0}, \delta_d)$ and $(X/\coherent{0}, \rho^{\Phi}).$ Now using \eqref{eqch4.9} and Definition~\ref{d1.2}, we obtain that $\Phi$ is a combinatorial self-similarity of $(X/\coherent{0}, \delta_d).$ Since $\Phi$ is an arbitrary element of $\mathbf{Sym} (X/\coherent{0}),$ the equality
\begin{equation*}
\mathbf{Cs}(X/\coherent{0}, \delta_d)=\mathbf{Sym} (X/\coherent{0})
\end{equation*}
holds. Consequently, by Theorem~\ref{Analogt2.4}, at least one from the following statements is valid:

$-$ $(X, d)$ is strongly rigid;

$-$ $(X, d)$ is discrete;

$-$ $(X, d)$ is a pseudorectangle.

For convenience we denote by $\mathcal{CL}_{st}$ and $\mathcal{CL}_{di}$ the classes of all strongly rigid pseudometroc spaces and all discrete ones respectively. In addition, write $\mathcal{CL}_{pr}$ for the class of all pseudorectangles and $\mathcal{CL}_{st}^{4}$ for the class of all $(X, d)\in\mathcal{CL}_{st}$ satisfying the inequality $|d(X^{2})|\leqslant 4.$

It was shown above that

\begin{equation}\label{eqch4.11}
\mathcal{CL}^{*}\subseteq \mathcal{CL}_{st}\cup\mathcal{CL}_{di}\cup\mathcal{CL}_{pr}.
\end{equation}

Let us prove that $(i_1) - (i_3)$ hold with $\mathcal{CL}=\mathcal{CL}^{**}$ for every $(X, d)~\in~\mathcal{CL}^{**}$ and every pseudometric space $(Y, \rho),$ if
\begin{equation}\label{eqch4.10}
\mathcal{CL}^{**}=\mathcal{CL}_{st} \quad\mbox{or}\quad \mathcal{CL}^{**}=\mathcal{CL}_{di}, \quad\mbox{or}\quad \mathcal{CL}^{**}=\mathcal{CL}_{pr}\cup\mathcal{CL}_{st}^{4}.
\end{equation}

Let \eqref{eqch4.10} hold and let $(X, d)\in\mathcal{CL}^{**}.$ Let us consider an arbitrary pseudometric space $(Y, \rho)$. If $(Y, \rho)$ and $(X, d)$ are pseudoisometric, then the metric reflection $(Y/\coherent{0(\rho)}, \delta_{\rho})$ and $(X/\coherent{0(d)}, \delta_{d})$ are isometric (see Remark~\ref{remnew}). By Lemma~\ref{l4.xx} we have $(X/\coherent{0(d)}, \delta_{d})\in\mathcal{CL}^{**}.$ Since $(Y/\coherent{0(\rho)}, \delta_{\rho})$ and $(X/\coherent{0(d)}, \delta_{d})$ are isometric, and \eqref{eqch4.10} holds, we obtain
$(Y/\coherent{0(\rho)}, \delta_{\rho})\in\mathcal{CL}^{**}$ directly from the definitions of $\mathcal{CL}_{st},$ $\mathcal{CL}_{di}$ and $\mathcal{CL}_{pr}.$ Hence, $(Y, \rho)\in\mathcal{CL}^{**}$ by Lemma~\ref{l4.xx}. Thus, condition $(i_1)$ is fulfilled.

If $(Y, \rho)\in\mathcal{CL}^{**}$ and $Y=X$ hold, then the identical mapping $Id_{X}: X\to X$ is a combinatorial similarity of $(X, d)$ and $(X, \rho)$ by Lemma~\ref{Lemch4}. Thus, $(i_2)$ is also valid. The validity of $(i_3)$ follows directly from: Definition~\ref{d4.3}, if $\mathcal{CL}^{**}=\mathcal{CL}_{st};$ Definition~\ref{defdis}, if $\mathcal{CL}^{**}=\mathcal{CL}_{di};$ Definition~\ref{defrect},  if $\mathcal{CL}^{**}=\mathcal{CL}_{pr}\cup\mathcal{CL}_{st}^{4}.$

We claim that at least one from the inclusions
\begin{equation}\label{eqch4.12}
\mathcal{CL}^{*}\subseteq\mathcal{CL}_{st}, \quad \mathcal{CL}^{*}\subseteq\mathcal{CL}_{di}, \quad \mathcal{CL}^{*}\subseteq\mathcal{CL}_{pr}\cup\mathcal{CL}_{st}^{4}.
\end{equation}
holds.

Let us denote by $\mathcal{CL}^{*}_{met}$ the subclass of all metric spaces which belong $\mathcal{CL}^{*}.$ Then, using \eqref{eqch4.11} and \eqref{eqch4.10}, and the definitions of the pseudorectangles, strongly rigid spaces and discrete spaces, we see that at least one from inclusions \eqref{eqch4.12} holds if and only if we have at least one from

\begin{equation}\label{eqch4.13}
\mathcal{CL}^{*}_{met}\subseteq\mathcal{CL}_{st}, \quad \mathcal{CL}^{*}_{met}\subseteq\mathcal{CL}_{di}, \quad \mathcal{CL}^{*}_{met}\subseteq\mathcal{CL}_{pr}\cup\mathcal{CL}_{st}^{4}.
\end{equation}

First of all we note that all inclusions in \eqref{eqch4.13} are simultaneously valid, if $|X|\leqslant 2$ holds for every $(X, d)\in \mathcal{CL}^{*}_{met}.$

Suppose that $|X|~\leqslant~3$ holds for every $(X, d)\in\mathcal{CL}^{*}_{met}$ and there is $(Y, \rho)\in \mathcal{CL}^{*}_{met}$ such that $|Y|=3.$ Then, using $(i_1)$ and $(i_2)$ for every $(Y_1, \rho_1)\in \mathcal{CL}^{*}_{met}$ with $|Y_1|=3$ we can find a metric $\rho^{*}$ on $Y$ such that $(Y, \rho^{*}),$ $(Y_1, \rho_1)$ are isometric and, in addition, $(Y, \rho)$ and $(Y, \rho^{*})$ are combinatorially similar. It implies the equality
\begin{equation*}\label{eqch4.14}
|\rho(Y^{2})|=|\rho_{1}(Y_{1}^{2})|.
\end{equation*}
The last equality and \eqref{eqch4.11} imply that both $(Y, \rho^{*})$ and $(Y_1, \rho_1)$ are either strongly rigid or discrete. Consequently, we have
\begin{equation}\label{eqch4.15}
\mbox{either} \quad \mathcal{CL}^{*}_{met}\subseteq\mathcal{CL}_{st}, \quad \mbox{or} \quad \mathcal{CL}^{*}_{met}\subseteq\mathcal{CL}_{di}.
\end{equation}

Let us consider now the case, when there is $(Y, \rho)\in\mathcal{CL}^{*}_{met}$ such that $|Y|=4$ and $|X|\leqslant 4$ holds for every $(X, d)\in\mathcal{CL}^{*}_{met}.$

If there is $(Y, \rho)\in\mathcal{CL}^{*}_{met}$ such that $|Y|=4$ and $(Y, \rho)$ is discrete (strongly rigid) then, as in the case $|Y|=3,$ we obtain that every $(Y_1, \rho_1)\in\mathcal{CL}^{*}_{met}$ with $|Y_1|=4$ is discrete (strongly rigid). Moreover, every three-point subspace of $(Y, \rho)$ is also discrete (strongly rigid) and, consequently, using $(i_1) - (i_2),$ we can prove that all $(Z, d)\in\mathcal{CL}^{*}_{met}$ with $|Z|=3$ are also discrete (strongly rigid). Thus, \eqref{eqch4.15} holds.

Suppose that $(Y, \rho)\in\mathcal{CL}^{*}_{met}$ such that $|Y|=4$ but $(Y, \rho)$ is neither discrete nor strongly rigid. Then \eqref{eqch4.11} implies that $(Y, \rho)$ is a pseudorectangle. Then, using $(i_2)$ and Proposition~\ref{p4.9}, we obtain that every $(Y_1, \rho_1)\in\mathcal{CL}^{*}_{met}$ with $|Y_1|=4$ is also a pseudorectangle. Since every three-point subspace of $(Y, \rho)$ is strongly rigid, we can prove that all $(Z, d)\in\mathcal{CL}^{*}_{met}$ with $|Z|=3$ are also strongly rigid. Thus,
\begin{equation*}
\mathcal{CL}^{*}_{met}\subseteq\mathcal{CL}_{pr}\cup\mathcal{CL}_{st}^{4}.
\end{equation*}

To complete the proof that at least one inclusion from \eqref{eqch4.13} holds it suffices to show that \eqref{eqch4.15} holds if $\mathcal{CL}^{*}_{met}$ contains a metric space $(Y, \rho)$ with $|Y|\geqslant 5.$ The last inequality implies that $(Y, \rho)$ is not a pseudorectangle. Hence, by \eqref{eqch4.11}, $(Y, \rho)$ is either discrete or strongly rigid.

If $(Y, \rho)$ is discrete (strongly rigid), then arguing as above, we obtain that every
$(Y_1, \rho_1)\in\mathcal{CL}^{*}_{met}$ with $|Y_1|\leqslant |Y|$ is also discrete (strongly rigid). If $(Z, d)\in\mathcal{CL}^{*}_{met}$ and $|Z|>|Y|$ holds, then the discretness (the strong rigidity) of $(Z, d)$ implies the discretness (the strong rigidity) of $(Y, \rho).$ Thus, \eqref{eqch4.15} holds.

Let us prove now that the classes $\mathcal{CL}_{st},\, \mathcal{CL}_{di}$ and $\mathcal{CL}_{pr}\cup \mathcal{CL}_{st}^{4}$ are maximal in sense of Remark~\ref{remch4}. To see it suppose that $\mathcal{CL}^{o}$ is a class of nonempty pseudometric spaces such that
\begin{equation}\label{eqch4.16}
\mathcal{CL}^{o}\supseteq \mathcal{CL}_{st}
\end{equation}
and $(i_1) - (i_3)$ hold for every $(X, d)\in\mathcal{CL}^{o}$ and each pseudometric space $(Y, \rho).$ Then, using \eqref{eqch4.12} with $\mathcal{CL}^{*}=\mathcal{CL}^{o}$ we obtain at least one from the inclusions
\begin{equation}\label{eqch4.17}
\mathcal{CL}^{o}\subseteq\mathcal{CL}_{st}, \quad \mathcal{CL}^{o}\subseteq\mathcal{CL}_{di}, \quad \mathcal{CL}^{o}\subseteq\mathcal{CL}_{pr}\cup \mathcal{CL}_{st}^{4}.
\end{equation}

Now $\mathcal{CL}_{st}\not\subseteq\mathcal{CL}_{di}$ and $\mathcal{CL}_{st}\not\subseteq\mathcal{CL}_{pr}\cup \mathcal{CL}_{st}^{4}$ together with \eqref{eqch4.16} and \eqref{eqch4.17} imply that $$\mathcal{CL}_{st}\supseteq\mathcal{CL}^{o}\supseteq \mathcal{CL}_{st}.$$ Thus, \eqref{eqch4.16} implies that $\mathcal{CL}^{o} = \mathcal{CL}_{st},$ i.e. $\mathcal{CL}_{st}$ is maximal.

The maximality of $\mathcal{CL}_{di}$ and $\mathcal{CL}_{pr}\cup\mathcal{CL}_{st}^{4}$ can be proved similarly. $\Box$

\begin{corollary}
Let $(Z, l)$ be a nonempty pseudometric space. Then the equality
\begin{equation*}
\mathbf{Cs}(Z/\coherent{0}, \delta_l)=\mathbf{Sym} (Z/\coherent{0})
\end{equation*}
holds, if and only if there is a class $\mathcal{CL}$ of nonempty pseudometric spaces, which satisfies $(Z, l) \in \mathcal{CL}$ and conditions $(i_1) - (i_3)$ of Theorem~\ref{ThCh4_4.8} for every $(X, d) \in \mathcal{CL}$ and each nonempty pseudometric space $(Y, \rho).$
\end{corollary}

In connection with Theorem~\ref{ThCh4_4.8}, the following question naturally arises. Are conditions $(i_1) - (i_3)$ independent of each other?

\section*{Funding}

Viktoriia Bilet was partially supported by the Grant EFDS-FL2-08 of the found The European Federation of  Academies of Sciences and Humanities (ALLEA) and by a grant from the Simons Foundation (Award 1160640, Presidential Discretionary-Ukraine Support Grants).

Oleksiy Dovgoshey was supported by Finnish Society of Sciences and Letters.

\section*{Acknowledgment}

The authors grateful to the referees for their valuable suggestions.

\bibliographystyle{vancouver}

\bibliography{bib2021.04}

\begin{thebibliography}{10}

\bibitem{Dov2019IEJA}
Dovgoshey O.
\newblock {Semigroups generated by partitions}.
\newblock Int Electron J Algebra. 2019;26:145--190.

\bibitem{DovBBMSSS2020}
Dovgoshey O.
\newblock Combinatorial properties of ultrametrics and generalized
  ultrametrics.
\newblock Bull Belg Math Soc Simon Stevin. 2020;27(3):379--417.

\bibitem{DLAMH2020}
Dovgoshey O, Luukkainen J.
\newblock Combinatorial characterization of pseudometrics.
\newblock Acta Math Hungar. 2020;161(1):257--291.

\bibitem{BD2206}
Bilet V, Dovgoshey O.
\newblock Completeness, closedness and metric reflections of pseudometric
  spaces.
\newblock Topology Appl, Article ID 108440,. 2023;327:14 p.

\bibitem{BD2205}
Bilet V, Dovgoshey O.
\newblock When all permutations are combinatorial similarities.
\newblock Bull Korean Math Soc. 2023;60(3):733--746.

\bibitem{Kur1934CRASP}
Kurepa D.
\newblock Tableaux ramifi\'{e}s d'ensemples, espaces pseudodistacies.
\newblock C R Acad Sci Paris. 1934;198:1563--1565.

\bibitem{Kelley1965}
Kelley JL.
\newblock {General Topology}.
\newblock New York --- Heidelberg --- Berlin: Springer-Verlag; 1975.

\bibitem{Broughan73}
Broughan KA.
\newblock A metric characterizing $\check{C}$ech dimension zero.
\newblock Proc Amer Math Soc. 1973;39:437--440.

\bibitem{Hattori90}
Hattori Y.
\newblock Congruence and dimension of nonseparable metric spaces,.
\newblock Proc Amer Math Soc. 1990;108(4):1103--1105.

\bibitem{Janos1972}
Janos L.
\newblock A metric characterization of zero-dimensional spaces.
\newblock {Proc Amer Math Soc}. 1972;31(1):268--270.

\bibitem{JanosMartin78}
Janos L, Martin H.
\newblock Metric characterizations of dimension for separable metric spaces,.
\newblock {Proc Amer Math Soc}. 1978;70(2):209--212.

\bibitem{Ishiki2022}
Ishiki Y.
\newblock Strongly rigid metrics in spaces of metrics.
\newblock arXiv:221002170v4. 2022;p. 1--22.

\bibitem{Martin1977}
Martin HW.
\newblock Strongly rigid metrics and zero dimensionality.
\newblock Proc Amer Math Soc. 1977;67(1):157--161.

\bibitem{BDKP2017AASFM}
Bilet V, Dovgoshey O, K\"{u}\c{c}\"{u}kaslan M, Petrov E.
\newblock {Minimal universal metric spaces}.
\newblock Ann Acad Sci Fenn Math. 2017;42(2):1019--1064.

\bibitem{DS2021aa}
Dovgoshey O, Shanin R.
\newblock Uniqueness of best proximity pairs and rigidity of semimetric spaces.
\newblock J Fixed Point Theory Appl, Paper No 34,. 2023;25(1):31 p.

\bibitem{Rouyer2011}
Rouyer J.
\newblock {Generic properties of compact metric spaces}.
\newblock Topology Appl. 2011;158(16):2140--2147.

\bibitem{Sea2007}
Searc\'{o}id MO.
\newblock {Metric Spaces}.
\newblock London: Springer---Verlag; 2007.

\bibitem{KurMost}
Kuratowski K, Mostowski A.
\newblock {Set Theory with an Introduction to Descriptive Set Theory}.
\newblock Amsterdam---New York---Oxford: North-Holland Publishing Company;
  1976.

\bibitem{Ore1942}
Ore O.
\newblock Theory of equivalence relations.
\newblock Duke Math J. 1942;9(3):573--627.

\end{thebibliography}

\end{document}